\documentclass[11pt, oneside]{article} %
\usepackage{graphicx}
\usepackage{amsmath,amssymb}%
\usepackage{amsthm}
\newtheorem{theorem}{Theorem}%
\newtheorem{lemma}[theorem]{Lemma}%
\newtheorem{corollary}{Corollary}[theorem]%
\theoremstyle{definition}%
\newtheorem{definition}{Definition}%
\theoremstyle{remark}%
\newtheorem{remark}{Remark}%
\newtheorem{example}{Example}%

\title{Parametrizations of minimal timelike 
surfaces in the four-dimensional pseudo-Euclidean space of index two%
\footnote{Dedicated to Professor Yoshihiro Ohnita on the occasion of his retirement.}}%
\author{Katsuhiro Moriya\\Graduate School of Science, University of Hyogo,\\
2167 Shosha, Himeji, Hyogo, 671-2280, JAPAN\\
m905k019@guh.u-hyogo.ac.jp}
\date{Keywords: timelike surfaces; minimal surfaces; parametrizations}
\begin{document}
\maketitle
\begin{center}
{\scriptsize
This is the author's accepted manuscript (postprint) 
of a paper accepted for publication in the Hokkaido Mathematical Journal. 
The final published version will be available from the publisher.}
\end{center}
\abstract{
We construct representation formulas for local null curves in the four-dimensional pseudo-Euclidean space of index two and derive corresponding parametrizations for local minimal timelike surfaces without integration. 
As a special case of the representation formula, we construct a representation formula for local null curves in the three-dimensional pseudo-Euclidean space of index one that involves integration. 
Our results  provide examples of minimal timelike surfaces.}

\section{Introduction}

A pseudo-Euclidean space is a finite dimensional real vector space equipped with an indefinite inner product. 
The index of a pseudo-Euclidean space is the dimension of the maximal linear subspace spanned by vectors of negative scalar square. 
In this paper, the $n$-dimensional pseudo-Euclidean space of index $r$ is denoted by $\mathbb{E}^n_r$. 
The $n$-dimensional pseudo-Euclidean space of index $0$ is the $n$-dimensional Euclidean space $\mathbb{E}^n$. 
Pseudo-Euclidean spaces play important roles in differential geometry and physics. 
For example, spacetime in the theory of special relativity is the Minkowski space, which is $\mathbb{E}^4_1$. 

A timelike surface in a pseudo-Euclidean space is a surface where an indefinite inner product of index one is induced to each tangent vector space.  
A mean curvature vector of a timelike surface is a measure of extrinsic curvature. 
A minimal timelike surface is a timelike surface with vanishing mean curvature vector. 
In physics, minimal timelike surfaces in the Minkowski space is the most studied case. However, in mathematics, it is important to develop a theory of minimal timelike surfaces in all pseudo-Euclidean spaces, regardless of dimension or index.

While finding parameterizations of minimal timelike surfaces might be a natural first approach, it is generally a challenging problem. 
Obtaining the parameterizations requires a deeper understanding of the underlying structure of minimal timelike surfaces. 
Null curves in $\mathbb{E}^4_2$ play an important role in providing a parametrization.
A null curves in $\mathbb{E}^4_2$ is a curve with null tangent vector. 
B. Y. Chen referred to minimal timelike surfaces as Lorentzian minimal surfaces and provided the following classification.
\begin{theorem}[Chen \cite{zbMATH05568272}]\label{thm:Chen}
Let $\langle\enskip,\enskip\rangle$ be the inner product of $\mathbb{E}^4_2$ and $z(x)$ and $w(y)$ be two null curves defined on open intervals
$I_1$ and $I_2$ respectively in $\mathbb{E}^4_2$. 
If $\langle z(x), w(y)\rangle \neq  0$ for
$(x, y) \in I_1 \times  I_2$, then
\begin{align*}
\psi(x, y) = z(x) + w(y)
\end{align*}
defines a Lorentzian minimal surface in $\mathbb{E}^4_2$.
Conversely, locally every Lorentzian minimal surface in $\mathbb{E}^4_2$ is congruent to the
translation surface defined above.
\end{theorem}
Here, Chen allowed surfaces to have points where they are not immersed, and still referred to them as surfaces. 
Points that are not immersed are precisely those where the two vectors $z'(x)$ and $w'(y)$ are linearly dependent. 
Following Chen, we refer to a surface that admits points which are not immersed simply as a \textit{surface}.
Then, the problem of giving a parametrization of a minimal timelike surface reduces to the problem of giving a parametrization of a null curve.

In this paper, a parametrization of a local null curve in $\mathbb{E}^4_2$ is constructed explicitly
using four real functions and their derivatives:
\begin{theorem}\label{thm:repnc}

Let $P_{ij}(\xi)$ $(i, j = 1, 2)$ be real-valued functions defined on an open interval $I$.  
Set  
\begin{align*}
p_1(\xi) &=
\begin{pmatrix}
P_{11}(\xi) & P_{12}(\xi)
\end{pmatrix}, \quad 
p_2(\xi) =
\begin{pmatrix}
P_{21}(\xi) & P_{22}(\xi)
\end{pmatrix}, \\
E &=
\begin{pmatrix}
1 & 0\\
0 & -1
\end{pmatrix}, \qquad
L =
\begin{pmatrix}
0 & 1\\
1 & 0
\end{pmatrix}.
\end{align*}
Assume that
\begin{align*}
\det
\begin{pmatrix}
p_2(\xi)\\
p_2'(\xi)
\end{pmatrix}
\neq 0.
\end{align*}
Then the map $\beta(\xi)$ defined by
\begin{align}
\beta(\xi) &=
\begin{pmatrix}
\beta_1(\xi)\\
\beta_2(\xi)\\
\beta_3(\xi)\\
\beta_4(\xi)
\end{pmatrix}
=
\frac{1}{4\det\begin{pmatrix}p_2(\xi)\\ p_2'(\xi)\end{pmatrix}}
\begin{pmatrix}
\phi(p_1(\xi), p_2(\xi))\\
\phi(p_1(\xi)E, p_2(\xi)L)\\
\phi(p_1(\xi)E, p_2(\xi))\\
\phi(p_1(\xi), p_2(\xi)L)
\end{pmatrix},
\label{eq:b}
\end{align}
is a null curve in $\mathbb{E}^4_2$.  
Here the map $\phi$ for two $\mathbb{R}^2$-valued functions  
$x(\xi)=\begin{pmatrix}x_1(\xi) & x_2(\xi)\end{pmatrix}$ and  
$y(\xi)=\begin{pmatrix}y_1(\xi) & y_2(\xi)\end{pmatrix}$ is defined by
\begin{align*}
\phi(x(\xi), y(\xi))
= -\det
\begin{pmatrix}
x(\xi)\\[1mm]
y'(\xi)
\end{pmatrix}
+\det
\begin{pmatrix}
x'(\xi)\\[1mm]
y(\xi)
\end{pmatrix}.
\end{align*}

Conversely, any given null curve $\beta(\xi)$ in $\mathbb{E}^4_2$ can be expressed in the form \eqref{eq:b},  
where the functions $P_{11}(\xi)$, $P_{12}(\xi)$, $P_{21}(\xi)$, $P_{22}(\xi)$ are given by
\begin{align}
\begin{aligned}
P_{11}(\xi)
&= -2\left\{(\beta_{1}(\xi) - \beta_{3}(\xi))(\beta_{2}'(\xi) + \beta_{4}'(\xi))\right.\\
&\left.-(\beta_{2}(\xi) + \beta_{4}(\xi))(\beta_{1}'(\xi) - \beta_{3}'(\xi))\right\}k(\xi),\\
P_{12}(\xi)
&= -2\left\{(\beta_{2}(\xi) - \beta_{4}(\xi))(\beta_{2}'(\xi) + \beta_{4}'(\xi))\right.\\
&\left.+(\beta_{1}(\xi) + \beta_{3}(\xi))(\beta_{1}'(\xi) - \beta_{3}'(\xi))\right\}k(\xi),\\
P_{21}(\xi)
&= (\beta_{2}'(\xi) + \beta_{4}'(\xi))k(\xi),\\
P_{22}(\xi)
&= (\beta_{1}'(\xi) - \beta_{3}'(\xi))k(\xi),
\end{aligned}
\label{eq:P}
\end{align}
and $k(\xi)$ is an arbitrary real function.
\end{theorem}
From Theorem \ref{thm:Chen} and Theorem \ref{thm:repnc}, we obtain the following corollary which provides a parametrization of a local minimal timelike surface in $\mathbb{E}^4_2$. 
\begin{corollary}\label{cor:repmt}
Let $\gamma_1(\xi_1)$ and $\gamma_2(\xi_2)$ be two null curves defined on open intervals $I_1$ and $I_2$, respectively, in $\mathbb{E}^4_2$, as given in the preceding theorem, and suppose that $\langle\gamma_1(\xi_1),\gamma_2(\xi_2)\rangle\neq 0$. Then the map $f\colon I_1\times I_2\to\mathbb{E}^4_2$ defined by 
$f (\xi_1,\xi_2) =\gamma_1(\xi_1) +\gamma_2(\xi_2)$ defines a minimal timelike surface on $I_1 \times I_2$. 
\end{corollary}
While this corollary can be deduced easily as a paraquaternion variation of a result in \cite{MR2496508}, we provide a direct computational proof here.

We compare our results with those of previous results. 
Kassabov and Milousheva \cite{MR4169475}, Kanchev, Kassabov and Milousheva \cite{MR4369193}, Budinichi and Rigoli \cite{MR0996195}, 
Konderak \cite{MR2141751} and 
Dussan, Franco Filho and Magid \cite{MR3673665} presented various methods for obtaining parameterizations. 
These methods are variants of 
the Weierstrass-Enneper representation formula for minimal surfaces in $\mathbb{E}^3$ by Weierstrass \cite{Weierstrass1866} and Enneper \cite{Enneper1864} 
and that in $\mathbb{E}^4$ by Eisenhardt \cite{zbMATH02630736,zbMATH02626848}. 
In contrast to previous studies, this paper presents a parameterization for minimal timelike surfaces that avoids the use of integration. 
This approach offers a more direct reflection of the structure of such surfaces.

In the case $\beta_4(\xi)=0$, 
the null curve $\beta(\xi)$ becomes a null curve in $\mathbb{E}^3_1$. 
Under this restriction, Theorem \ref{thm:repnc} reduces to the following corollary 
that provides a parametrization of null curves in 
$\mathbb{E}^3_1$. 
\begin{corollary}\label{cor:null3}
Let $P_{12}(\xi)$, $P_{21}(\xi)$ and $P_{22}(\xi)$ be real functions on an open interval $I$
with
\begin{align*}
&
\det\begin{pmatrix}
P_{21}(\xi)&P_{22}(\xi)\\
P'_{21}(\xi)&P_{22}'(\xi)
\end{pmatrix}
\neq 0,
\quad
P_{21}(\xi)\neq 0.
\end{align*}
Then, the map 
$\beta(\xi)$ defined as follows is a null curve in $\mathbb{E}^3_1$ with non-constant $\beta_2(\xi)$;
\begin{align}
\begin{aligned}
&P_{11}(\xi)=P_{21}(\xi)\left(\int_{\xi_0}^\xi\frac{P_{12}'(\xi)P_{22}(\xi)-P_{12}(\xi)P_{22}'(\xi)}{(P_{21}(\xi))^2}d\xi+C\right)\\
&
\beta(\xi)=
\begin{pmatrix}
\beta_1(\xi)\\
\beta_2(\xi)\\
\beta_3(\xi)
\end{pmatrix}
=
\frac{1}{4\det\begin{pmatrix}p_2(\xi)\\p_2'(\xi)\end{pmatrix}}
\begin{pmatrix}
\phi(p_1(\xi), p_2(\xi))\\
\phi(p_1(\xi)E, p_2(\xi)L)\\
\phi(p_1(\xi)E, p_2(\xi))
\end{pmatrix},
\end{aligned}
\label{eq:b3}
\end{align}
where $\xi_0$ is an arbitrary point in $I$, and $C$ is an arbitrary constant. 

Conversely, any given null curve $\beta(\xi)$ in $\mathbb{E}^3_1$ with non-constant $\beta_2(\xi)$ can be expressed in the form \eqref{eq:b3},  
where the functions $P_{12}(\xi)$, $P_{21}(\xi)$, $P_{22}(\xi)$ are given by
\begin{align*}
P_{12}(\xi)=&-2\{\beta_{2}(\xi)\beta_{2}'(\xi)+(\beta_1(\xi)+\beta_3(\xi))(\beta_{1}'(\xi)-\beta_{3}'(\xi))\}k(\xi),\\
P_{21}(\xi)=&\beta_{2}'(\xi)k(\xi),\\
P_{22}(\xi)=&(\beta_{1}'(\xi)-\beta_{3}'(\xi))k(\xi) 
\end{align*}
and $k(\xi)$ is an arbitrary real function.
\end{corollary}
From Theorem \ref{thm:Chen} and Corollary \ref{cor:null3}, we obtain the following corollary which provides a parametrization of a minimal timelike surface in $\mathbb{E}^3_1$. 
\begin{corollary}\label{cor:repmt3}
Let $\gamma_1(\xi_1)$ and $\gamma_2(\xi_2)$ be two null curves  defined on open intervals $I_1$ and $I_2$, respectively, in $\mathbb{E}^3_1$, as given in the preceding corollary, 
and suppose that $\langle\gamma_1(\xi_1),\gamma_2(\xi_2)\rangle\neq 0$. 
Then the map $f\colon I_1\times I_2\to\mathbb{E}^3_1$ defined by 
$f(\xi_1,\xi_2) =\gamma_1(\xi_1) + \gamma_2(\xi_2)$ 
defines a minimal timelike surface on $I_1 \times I_2$ with non-constant $f_2$, where 
$f_2$ denotes the second coordinate function of 
$f$. 
\end{corollary}

We construct examples of  null curves using  Theorem~\ref{thm:repnc} and Corollary~\ref{cor:null3} and minimal timelike surfaces using Corollary~\ref{cor:repmt} and Corollary~\ref{cor:repmt3}.

\section{Null curves and minimal timelike surfaces}
Throughout this paper, we assume that 
maps are smooth, unless explicitly mentioned otherwise. 
In this section, we review null curves and minimal timelike surfaces. 

Let $\mathbb{R}$ be the algebra of real numbers and $\mathbb{R}^n=\{(x_1,x_2,\ldots,x_{n}):x_i\in\mathbb{R}\quad(i=1,2,\ldots,n)\}$. 
The $4$-dimensional pseudo-Euclidean space with index $2$ $\mathbb{E}^4_2$ is  $\mathbb{R}^{4}$ 
with indefinite inner product $\langle\enskip,\enskip\rangle$ defined by 
\begin{align*}
\langle (x_1,x_2,x_3,x_{4}),(y_1,y_2,y_3,y_{4})\rangle=x_1y_1+x_2y_2-x_3y_3-x_4y_4. 
\end{align*}
for $ (x_1,x_2,x_3,x_{4})$, $(y_1,y_2,y_3,y_{4})\in \mathbb{R}^4$. 

Let $I$ be an open interval and $u$ be a parameter of $I$. 
\begin{definition}
A map $c\colon I\to\mathbb{E}^4_2$ is called \textit{null} 
if 
\begin{align}
&\langle c_{u},c_{u}\rangle=0.\label{eq:nc}
\end{align}
\end{definition}

Let $U$ be a simply-connected open set of $\mathbb{R}^2$. 
\begin{definition}
A map $f\colon U\to\mathbb{E}^4_2$ is called a \textit{minimal timelike surface} 
if there is a coordinate $(\xi_1,\xi_2)$ of $U$ such that 
\begin{align}
&\langle f_{\xi_1},f_{\xi_1}\rangle=\langle f_{\xi_2},f_{\xi_2}\rangle=0,\quad f_{\xi_1\xi_2}=0. \label{eq:mts2}
\end{align}
\end{definition}
The coordinate $(\xi_1,\xi_2)$ is a null coordinate of the induced metric on $U$.  

\section{Proof of theorems}
Firstly, we prove a lemma for the calculation later. 
Put
\begin{align*}
&\phi_n(x(\xi),y(\xi))=
(\phi(x(\xi),y(\xi)))^{(n)}\quad (n=0,1),\\
&X(\xi)=\begin{pmatrix}x_1(\xi)&y_2(\xi)\end{pmatrix},\quad Y(\xi)=\begin{pmatrix}y_1(\xi)&x_2(\xi)\end{pmatrix},\\
&U(\xi)=\begin{pmatrix}x_1(\xi)&y_1(\xi)\end{pmatrix},\quad V(\xi)=\begin{pmatrix}y_2(\xi)&x_2(\xi)\end{pmatrix}.
\end{align*}
Then 
\begin{align*}
&\phi_n(x(\xi),y(\xi))\\
&=- \det\begin{pmatrix}x(\xi)\\y^{(n+1)}(\xi)\end{pmatrix} + \det\begin{pmatrix}x^{(n+1)}(\xi)\\ y(\xi)\end{pmatrix}\\
&=- \det\begin{pmatrix}X(\xi)\\X^{(n+1)}(\xi)\end{pmatrix} + \det\begin{pmatrix}Y^{(n+1)}(\xi)\\ Y(\xi)\end{pmatrix},\\
&\phi_n(x(\xi),y(\xi)L)\\
&=- \det\begin{pmatrix}U(\xi)\\U^{(n+1)}(\xi)\end{pmatrix} + \det\begin{pmatrix}V^{(n+1)}(\xi)\\ V(\xi)\end{pmatrix}. 
\end{align*}
\begin{lemma}\label{lem:phi}
The following identities hold:
\begin{align}
&
\begin{aligned}
&(\phi_n(x(\xi),y(\xi)))^2+(\phi_n(x(\xi)E,y(\xi)L))^2\\
&-(\phi_n(x(\xi)E,y(\xi)))^2-(\phi_n(x(\xi),y(\xi)L))^2\\
&=-4\det\begin{pmatrix}x(\xi)\\x^{(n+1)}(\xi)\end{pmatrix}\det\begin{pmatrix}y(\xi)\\y^{(n+1)}(\xi)\end{pmatrix},
\end{aligned}
\label{lem:eq1}
\\
&
\begin{aligned}
&\phi_1(x(\xi),y(\xi))\phi(x(\xi),y(\xi))+\phi_1(x(\xi)E,y(\xi)L)\phi(x(\xi)E,y(\xi)L)\\
&-\phi_1(x(\xi)E,y(\xi))\phi(x(\xi)E,y(\xi))-\phi_1(x(\xi),y(\xi)L)\phi(x(\xi),y(\xi)L)\\
&=-2\det\begin{pmatrix}x(\xi)\\x''(\xi)\end{pmatrix}\det\begin{pmatrix}y(\xi)\\y'(\xi)\end{pmatrix}-2\det\begin{pmatrix}x(\xi)\\x'(\xi)\end{pmatrix}\det\begin{pmatrix}y(\xi)\\y''(\xi)\end{pmatrix}.
\end{aligned}
\label{lem:eq2}
\end{align}
\end{lemma}
\begin{proof}
We have 
\begin{align*}
\phi_n&(x(\xi),y(\xi))^2-\phi_n(x(\xi)E,y(\xi))^2\\
=&(\phi_n(x(\xi),y(\xi))-\phi_n(x(\xi)E,y(\xi)))\\
&\times(\phi_n(x(\xi),y(\xi))+\phi_n(x(\xi)E,y(\xi)))\\
=&\left(-2x_1(\xi)y_2^{(n+1)}(\xi)+2x_1^{(n+1)}(\xi)y_2(\xi)\right)\\
&\times \left(2x_2(\xi)y_1^{(n+1)}(\xi)-2x_2^{(n+1)}(\xi)y_1(\xi)\right)\\
=&4\det\begin{pmatrix}X(\xi)\\X^{(n+1)}(\xi)\end{pmatrix}\det\begin{pmatrix}Y(\xi)\\Y^{(n+1)}(\xi)\end{pmatrix}. 
\end{align*}
Similarly, we have 
\begin{align*}
&-\phi_n(x(\xi),y(\xi)L)^2+\phi_n(x(\xi)E,y(\xi)L)^2\\
&=-4\det\begin{pmatrix}U(\xi)\\U^{(n+1)}(\xi)\end{pmatrix}\det\begin{pmatrix}V(\xi)\\V^{(n+1)}(\xi)\end{pmatrix}. 
\end{align*}
By adding these two equations term by term, we obtain \eqref{lem:eq1}. 
Differentiating both sides of this equation for $n=1$, we have 
\begin{align*}
&2\phi_1(x(\xi),y(\xi))\phi(x(\xi),y(\xi))+2\phi_1(x(\xi)E,y(\xi)L)\phi(x(\xi)E,y(\xi)L)\\
&-2\phi_1(x(\xi)E,y(\xi))\phi(x(\xi)E,y(\xi))-2\phi_1(x(\xi),y(\xi)L)\phi(x(\xi),y(\xi)L)\\
&=-4\det\begin{pmatrix}x(\xi)\\x^{(2)}(\xi)\end{pmatrix}\det\begin{pmatrix}y(\xi)\\y^{(1)}(\xi)\end{pmatrix}-4\det\begin{pmatrix}x(\xi)\\x^{(1)}(\xi)\end{pmatrix}\det\begin{pmatrix}y(\xi)\\y^{(2)}(\xi)\end{pmatrix}.
\end{align*}
This is equivalent to \eqref{lem:eq2}. 
\end{proof}

\begin{proof}[Proof of  Theorem~\ref{thm:repnc}]
The derivatives  of $\beta_1(\xi)$, $\beta_2(\xi)$, $\beta_3(\xi)$ and $\beta_4(\xi)$ are  
\begin{align*}
&\beta_1'(\xi)=
\frac{\phi_1(p_1(\xi), p_2(\xi))\det\begin{pmatrix}p_2(\xi)\\p_2'(\xi)\end{pmatrix}-\phi(p_1(\xi), p_2(\xi))\det\begin{pmatrix}p_2(\xi)\\p_2''(\xi)\end{pmatrix}}{4\det\begin{pmatrix}p_2(\xi)\\p_2'(\xi)\end{pmatrix}^2},\\
&\beta_2'(\xi)=
\frac{\phi_1(p_1(\xi)E, p_2(\xi)L)\det\begin{pmatrix}p_2(\xi)\\p_2'(\xi)\end{pmatrix}-\phi(p_1(\xi)E, p_2(\xi)L)\det\begin{pmatrix}p_2(\xi)\\p_2''(\xi)\end{pmatrix}}{4\det\begin{pmatrix}p_2(\xi)\\p_2'(\xi)\end{pmatrix}^2},\\
&\beta_3'(\xi)=
\frac{\phi_1(p_1(\xi)E, p_2(\xi))\det\begin{pmatrix}p_2(\xi)\\p_2'(\xi)\end{pmatrix}-\phi(p_1(\xi)E, p_2(\xi))\det\begin{pmatrix}p_2(\xi)\\p_2''(\xi)\end{pmatrix}}{4\det\begin{pmatrix}p_2(\xi)\\p_2'(\xi)\end{pmatrix}^2},\\
&\beta_4'(\xi)=
\frac{\phi_1(p_1(\xi), p_2(\xi)L)\det\begin{pmatrix}p_2(\xi)\\p_2'(\xi)\end{pmatrix}-\phi(p_1(\xi), p_2(\xi)L)\det\begin{pmatrix}p_2(\xi)\\p_2''(\xi)\end{pmatrix}}{4\det\begin{pmatrix}p_2(\xi)\\p_2'(\xi)\end{pmatrix}^2}. 
\end{align*}
By Lemma \ref{lem:phi}, we have 
\begin{align*}
&16\det\begin{pmatrix}p_2(\xi)\\p_2'(\xi)\end{pmatrix}^2\left((\beta'_1(\xi))^2+(\beta'_2(\xi))^2-(\beta'_3(\xi))^2-(\beta'_4(\xi))^2\right)\\
&=-4\det\begin{pmatrix}p_1(\xi)\\p_1''(\xi)\end{pmatrix}\det\begin{pmatrix}p_2(\xi)\\p_2''(\xi)\end{pmatrix}\det\begin{pmatrix}p_2(\xi)\\p_2'(\xi)\end{pmatrix}^2\\
&-2\left(-2\det\begin{pmatrix}p_1(\xi)\\p_1''(\xi)\end{pmatrix}\det\begin{pmatrix}p_2(\xi)\\p_2'(\xi)\end{pmatrix}-2\det\begin{pmatrix}p_1(\xi)\\p_1'(\xi)\end{pmatrix}\det\begin{pmatrix}p_2(\xi)\\p_2''(\xi)\end{pmatrix}\right)\\
&\times \det\begin{pmatrix}p_2(\xi)\\p_2'(\xi)\end{pmatrix}\det\begin{pmatrix}p_2(\xi)\\p_2''(\xi)\end{pmatrix}\\
&-4\det\begin{pmatrix}p_1(\xi)\\p_1'(\xi)\end{pmatrix}\det\begin{pmatrix}p_2(\xi)\\p_2'(\xi)\end{pmatrix}\det\begin{pmatrix}p_2(\xi)\\p_2''(\xi)\end{pmatrix}^2=0. 
\end{align*}
Hence, $\beta(\xi)$ is a null curve. 

For the proof of the converse, we see that the following equations hold from \eqref{eq:P}:
\begin{align*}
&(\beta_1'(\xi)-\beta_3'(\xi))P_{21}(\xi)-(\beta_2'(\xi)+\beta_{4}'(\xi))P_{22}(\xi)=0,\\
&-2\{(\beta_1(\xi)-\beta_3(\xi))P_{21}(\xi)-(\beta_2(\xi)+\beta_4(\xi))P_{22}(\xi)\}=P_{11}(\xi),\\
&-2\{(\beta_2(\xi)-\beta_4(\xi))P_{21}(\xi)+(\beta_1(\xi)+\beta_3(\xi))P_{22}(\xi)\}=P_{12}(\xi). 
\end{align*}
By differentiating both sides of the second and third equations, and using the first equation together with the null curve condition, we obtain 
\begin{align*}
&-2\{(\beta_1(\xi)-\beta_3(\xi))P_{21}'(\xi)-(\beta_2(\xi)+\beta_4(\xi))P_{22}'(\xi)\}=P_{11}'(\xi),\\
&-2\{(\beta_2(\xi)-\beta_4(\xi))P_{21}'(\xi)+(\beta_1(\xi)+\beta_3(\xi))P_{22}'(\xi)\}=P_{12}'(\xi). 
\end{align*}
Then 
\begin{align*}
&-
\det\begin{pmatrix}
P_{11}(\xi)&P_{12}(\xi)\\
P_{21}'(\xi)&P_{22}'(\xi)
\end{pmatrix}
\\
&=
2\{(\beta_1(\xi)-\beta_3(\xi))P_{21}(\xi)P_{22}'(\xi)-(\beta_2(\xi)+\beta_4(\xi))P_{22}(\xi)P_{22}'(\xi)\}\\
&-2\{(\beta_2(\xi)-\beta_4(\xi))P_{21}(\xi)P_{21}'(\xi)+(\beta_1(\xi)+\beta_3(\xi))P_{22}(\xi)P_{21}'(\xi)\},\\
&
\det\begin{pmatrix}
P_{11}'(\xi)&P_{12}'(\xi)\\
P_{21}(\xi)&P_{22}(\xi)
\end{pmatrix}
\\
&=-2\{(\beta_1(\xi)-\beta_3(\xi))P_{21}'(\xi)P_{22}(\xi)-(\beta_2(\xi)+\beta_4(\xi))P_{22}'(\xi)P_{22}(\xi)\}\\
&+2\{(\beta_2(\xi)-\beta_4(\xi))P_{21}'(\xi)P_{21}(\xi)+(\beta_1(\xi)+\beta_3(\xi))P_{22}'(\xi)P_{21}(\xi)\}.
\end{align*}
By adding these two equations term by term, we obtain
\begin{align*}
&-
\det\begin{pmatrix}
P_{11}(\xi)&P_{12}(\xi)\\
P_{21}'(\xi)&P_{22}'(\xi)
\end{pmatrix}
+
\det\begin{pmatrix}
P_{11}'(\xi)&P_{12}'(\xi)\\
P_{21}(\xi)&P_{22}(\xi)
\end{pmatrix}
\\
&=
4
\beta_1(\xi)
\det
\begin{pmatrix}
P_{21}(\xi)&P_{22}(\xi)\\
P_{21}'(\xi)&P_{22}'(\xi)
\end{pmatrix}
.
\end{align*}
Then, 
\begin{align*}
\beta_1(\xi)=\frac{1}{4\det\begin{pmatrix}p_1(\xi)\\p_2(\xi)\end{pmatrix}}\phi(p_1(\xi),p_2(\xi)). 
\end{align*}
The functions $\beta_2(\xi)$, $\beta_3(\xi)$ and $\beta_4(\xi)$ are obtained by similar calculation. 
\end{proof}
\begin{remark}
The null curve equation $(\beta_1'(\xi))^2+(\beta_2'(\xi))^2-(\beta_3'(\xi))^2-(\beta_4'(\xi))^2=0$ 
provides alternative representation of equation \eqref{eq:P}, for example, 
\begin{align*}
P_{11}(\xi)=&-2\{-(\beta_{1}(\xi)-\beta_{3}(\xi))(\beta_{1}'(\xi)+\beta_{3}'(\xi))\\
&-(\beta_2(\xi)+\beta_4(\xi))(\beta_{2}'(\xi)-\beta_{4}'(\xi))\}k(\xi),\\
P_{12}(\xi)=&-2\{(\beta_{2}(\xi)-\beta_{4}(\xi))(\beta_{2}'(\xi)+\beta_{4}'(\xi))\\
&+(\beta_1(\xi)+\beta_3(\xi))(\beta_{1}'(\xi)-\beta_{3}'(\xi))\}k(\xi),\\
P_{21}(\xi)=&-(\beta_{1}'(\xi)+\beta_{3}'(\xi))k(\xi),\\
P_{22}(\xi)=&(\beta_{2}'(\xi)-\beta_{4}'(\xi))k(\xi). 
\end{align*}
\end{remark}

\begin{proof}[Proof of Corollary \ref{cor:null3}]
In the formula in Theorem~\ref{thm:repnc}, $\beta_4=0$ is equivalent to the equation
\begin{align*}
&-
\det\begin{pmatrix}
P_{11}(\xi)&P_{12}(\xi)\\
P_{22}'(\xi)&P_{21}'(\xi)
\end{pmatrix}
+
\det\begin{pmatrix}
P_{11}'(\xi)&P_{12}'(\xi)\\
P_{22}(\xi)&P_{21}(\xi)
\end{pmatrix}
=0,
\end{align*}
We regard this equation as the differential equation
\begin{align*}
P_{21}(\xi)P_{11}'(\xi)-P_{21}'(\xi)P_{11}(\xi)=P_{12}'(\xi)P_{22}(\xi)-P_{12}(\xi)P_{22}'(\xi)
\end{align*}
with respect to $P_{11}(\xi)$. 
Since $P_{21}\neq 0$,  
\begin{align*}
P_{11}(\xi)=P_{21}(\xi)\left(\int_{\xi_0}^\xi\frac{P_{12}'(\xi)P_{22}(\xi)-P_{12}(\xi)P_{22}'(\xi)}{(P_{21}(\xi))^2}d\xi+C\right)
\end{align*}
is the solution to the equation. 
\end{proof}

\section{Examples}
We construct null curves by Theorem~\ref{thm:repnc} and Corollary \ref{cor:null3}
and construct minimal timelike surfaces by Corollary~\ref{cor:repmt} and Corollary~\ref{cor:repmt3}. 
\begin{example}
Let $p$, $q$, $r$, $s\in\mathbb{R}$ with $r\neq s$ and 
\begin{align*}
P_{11}(\xi)=e^{p\xi},\, P_{12}(\xi)=e^{q\xi},\, P_{21}(\xi)=e^{r\xi},\, P_{22}(\xi)=e^{s\xi}.  
\end{align*}
Then, 
\begin{align*}
&\alpha_1(\xi)
=
\frac{1}{4(s-r)}
\begin{pmatrix}
(p-s)e^{(p-r)\xi}+(r-q)e^{(q-s)\xi}\\
(p-r)e^{(p-s)\xi}+(q-s)e^{(q-r)\xi}\\
(p-s)e^{(p-r)\xi}-(r-q)e^{(q-s)\xi}\\
(p-r)e^{(p-s)\xi}-(q-s)e^{(q-r)\xi}
\end{pmatrix}
\end{align*}
is a null curve in $\mathbb{E}^4_2$. 
\end{example}
\begin{example}
Let $p$, $q\in\mathbb{R}$ with $q\neq 0$ and 
\begin{align*}
P_{11}(\xi)=\cos p\xi,\, P_{12}(\xi)=\sin p\xi,\, P_{21}(\xi)=\cos q\xi,\, P_{22}(\xi)=\sin q\xi.  
\end{align*}
Then, 
\begin{align*}
&\alpha_2(\xi)
=
\frac{1}{4q}
\begin{pmatrix}
-(p+q)\cos(p-q)\xi\\
-(p+q)\sin(p-q)\xi\\
(p-q)\cos(p+q)\xi\\
-(p-q)\sin(p+q)\xi
\end{pmatrix}
\end{align*}
is a null curve in $\mathbb{E}^4_2$. 
\end{example}
\begin{example}\label{example:e}
Let $q$, $r$, $s\in\mathbb{R}$ with $r\neq s$, $q+s-2r\neq 0$ and 
\begin{align*}
P_{12}(\xi)=e^{q\xi},\, P_{21}(\xi)=e^{r\xi},\, P_{22}(\xi)=e^{s\xi},\,\xi_0=0,\,C=0.  
\end{align*}
Then, 
\begin{align*}
P_{11}(\xi)=\frac{q-s}{q+s-2r}e^{(q+s-r)\xi}-\frac{q-s}{q+s-2r}e^{r\xi}
\end{align*}
and 
\begin{align*}
&\alpha_3(\xi)
=
\frac{1}{4}
\begin{pmatrix}
\alpha_{13}(\xi)\\
\alpha_{23}(\xi)\\
\alpha_{33}(\xi)
\end{pmatrix},\\
&\alpha_{13}(\xi)=\frac{(q-s)(q-r)}{(q+s-2r)(s-r)}e^{(q+s-2r)\xi}+\frac{q-s}{q+s-2r}+\frac{r-q}{s-r}e^{(q-s)\xi},\\
&\alpha_{23}(\xi)=\frac{2(q-s)}{s-r}e^{(q-r)\xi},\\
&\alpha_{33}(\xi)=\frac{(q-s)(q-r)}{(q+s-2r)(s-r)}e^{(q+s-2r)\xi}+\frac{q-s}{q+s-2r}-\frac{r-q}{s-r}e^{(q-s)\xi}
\end{align*}
is a null curve in $\mathbb{E}^3_1$. 
\end{example}
\begin{example}
Let $\alpha_4$ be the null curve where $q=4$, $r=2$, $s=1$ in Example~\ref{example:e}. 
Then, 
\begin{align*}
\alpha_4(\xi)=
\begin{pmatrix}
\displaystyle{-\frac{3}{2}e^{\xi}+\frac{3}{4}+\frac{1}{2}e^{3\xi}}\\[10pt]
\displaystyle{-\frac{3}{2}e^{2\xi}}\\[10pt]
\displaystyle{-\frac{3}{2}e^{\xi}+\frac{3}{4}-\frac{1}{2}e^{3\xi}}
\end{pmatrix}. 
\end{align*}
Figure \ref{fig:c2} depicts the curve determined by this parametrization. 
\begin{figure}[h]
  \centering
  \includegraphics[width=0.4\linewidth]{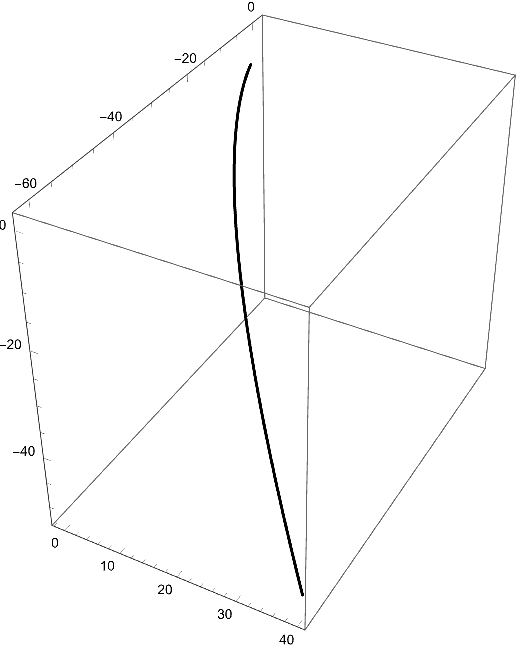}
  \caption{$\alpha_4$}
  \label{fig:c2}
\end{figure}

Let $\tilde{\alpha}_4$ be the null curve where $q=2$, $r=1$, $s=0$ in Example~\ref{example:e}. 
Then, 
\begin{align*}
\tilde{\alpha}_4(\xi)=
\begin{pmatrix}
\displaystyle{\frac{1}{6}e^{3\xi}+\frac{1}{12}-\frac{1}{2}e^{\xi}}\\[10pt]
\displaystyle{\frac{1}{2}e^{2\xi}}\\[10pt]
\displaystyle{\frac{1}{6}e^{3\xi}+\frac{1}{12}+\frac{1}{2}e^{\xi}}
\end{pmatrix}. 
\end{align*}
Figure~\ref{fig:c3} depicts the surface determined by this parametrization. 
\begin{figure}[h]
  \centering
  \includegraphics[width=0.4\linewidth]{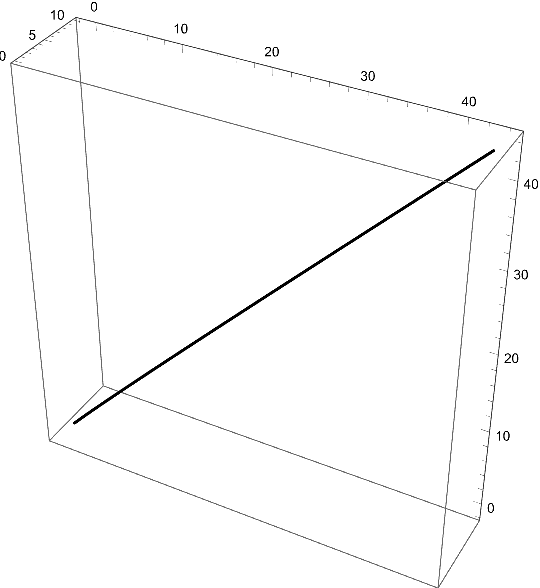}
  \caption{$\tilde{\alpha}_4$}
  \label{fig:c3}
\end{figure}
Then, the map $f_4\colon \mathbb{R}^2\to \mathbb{E}^3_1$, $f_4(\xi_1,\xi_2)=\alpha_4(\xi_1)+\tilde{\alpha}_4(\xi_2)$
is a minimal timelike surface. Figure~\ref{fig:s1} depicts the surface determined by this parametrization. 
\begin{figure}[h]
  \centering
  \includegraphics[width=0.4\linewidth]{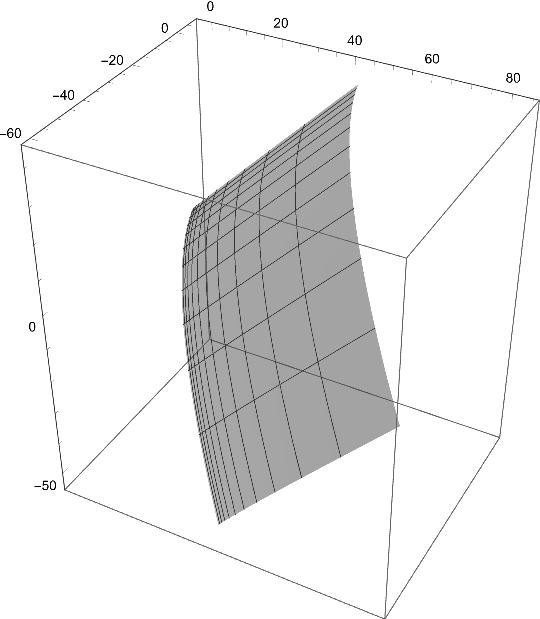}
  \caption{$f_4$}
  \label{fig:s1}
\end{figure}
\end{example}
\begin{example}
Let 
\begin{align*}
P_{12}(\xi)=\sin 2\xi,\, P_{21}(\xi)=\cos\xi,\, P_{22}(\xi)=\sin \xi,\,\xi_0=0,\,C=0.  
\end{align*}
Then, 
\begin{align*}
P_{11}(\xi)=-2(\cos\xi-1)^2
\end{align*}
and 
\begin{align*}
&\alpha_5(\xi)
=
\begin{pmatrix}
\displaystyle{-1+\frac{3}{2}\cos\xi-(\cos\xi)^3}\\[10pt]
-(\sin\xi)^3\\[10pt]
\displaystyle{\frac{3}{2}\cos\xi-1}
\end{pmatrix}
\end{align*}
is a null curve in $\mathbb{E}^3_1$. 
Figure \ref{fig:c1} depicts the curve determined by this equation
\begin{figure}[h]
  \centering
  \includegraphics[width=0.4\linewidth]{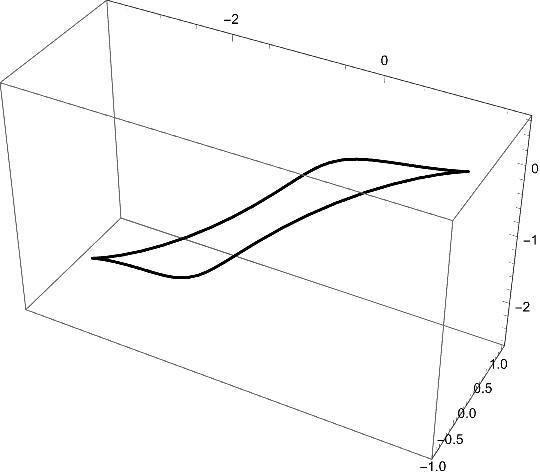}
  \caption{$\alpha_5$}
  \label{fig:c1}
\end{figure}

The curve 
\begin{align*}
&\tilde{\alpha}_5(\xi)
=
\begin{pmatrix}
(\sin\xi)^3\\[10pt]
\displaystyle{-1+\frac{3}{2}\cos\xi-(\cos\xi)^3}\\[10pt]
\displaystyle{\frac{3}{2}\cos\xi-1}
\end{pmatrix}
\end{align*}
is a null curve in $\mathbb{E}^3_1$. 
Figure~\ref{fig:c4} depicts the curve determined by this parametrization. 
\begin{figure}[h]
  \centering
  \includegraphics[width=0.4\linewidth]{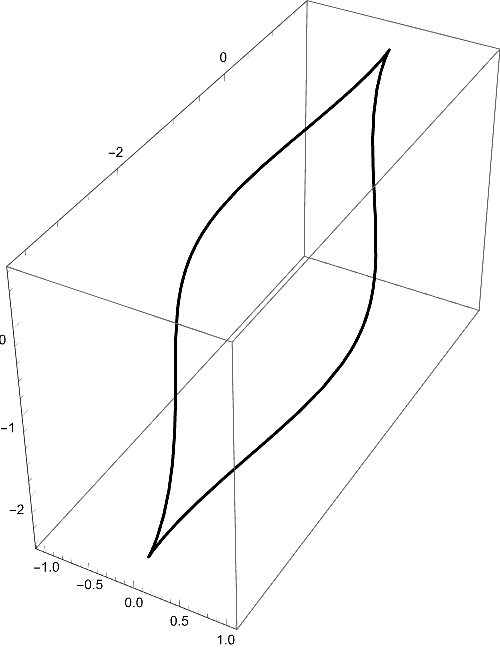}
  \caption{$\tilde{\alpha}_5$}
  \label{fig:c4}
\end{figure}
Then, the map $f_5\colon \mathbb{R}^2\to \mathbb{E}^3_1$, $f_5(\xi_1,\xi_2)=\alpha_5(\xi_1)+\tilde{\alpha}_5(\xi_2)$
is a minimal timelike surface. 
Figure~\ref{fig:s2} depicts the surface determined by this parametrization. 
\begin{figure}
  \centering
  \includegraphics[width=0.4\linewidth]{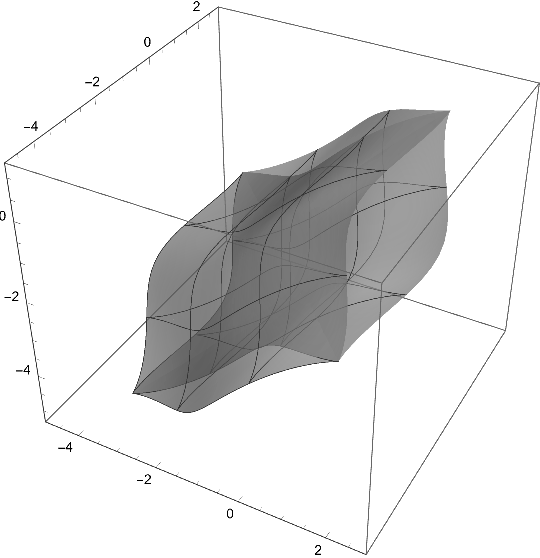}
  \caption{$f_5$}
  \label{fig:s2}
\end{figure}

\end{example}


\begin{thebibliography}{10}

\bibitem{MR0996195}
Budinich, P., Rigoli, M., Cartan spinors, minimal surfaces and strings.
\newblock Nuovo Cimento B (11) \textbf{102}(6) (1988), 609--647.

\bibitem{zbMATH05568272}
Chen, B.-Y., 
Nonlinear {Klein}-{Gordon} equations and {Lorentzian} minimal surfaces in {Lorentzian} complex space forms. 
\newblock Taiwanese J. Math.\textbf{13}(1) (2009), 1--24. 

\bibitem{MR3673665}
Dussan, M. P., Franco~Filho, A. P., Magid, M., The {B}j\"{o}rling problem for
  timelike minimal surfaces in {$\mathbb{R}^4_1$}.
\newblock Ann. Mat. Pura Appl. (4) \textbf{196}(4) (2017), 1231--1249.

\bibitem{zbMATH02630736}
Eisenhart, L. P.,  A fundamental parametric representation of space curves.
\newblock Ann. Math. (2) \textbf{13} (1911), 17--35.


\bibitem{zbMATH02626848}
Eisenhart, L. P., Minimal surfaces in euclidean four-space.
\newblock Amer. J. Math. \textbf{34} (1912), 215--236.

\bibitem{Enneper1864}
Enneper, A., Analytisch-geometrische untersuchungen.
\newblock Zeitschrift f\"{u}r Mathematik und Physik \textbf{9} (1864), 96--125. 

\bibitem{MR4369193}
Kanchev, K., Kassabov, O., Milousheva, V, Explicit solving of the system of
  natural {PDE}s of minimal {L}orentz surfaces in {$\mathbb{R}^4_2$}.
\newblock J. Math. Anal. Appl. \textbf{510}(1) (2022), Paper No. 126017.

\bibitem{MR4169475}
Kassabov, O., Milousheva, V, Weierstrass representations of {L}orentzian minimal
  surfaces in {$\mathbb{R}^4_2$}.
\newblock Mediterr. J. Math. \textbf{17}(6) (2020), Paper No. 199.

\bibitem{MR2141751}
Konderak, J. J, A {W}eierstrass representation theorem for {L}orentz surfaces.
\newblock Complex Var. Theory Appl. \textbf{50}(5) (2005), 319--332.

\bibitem{MR2496508}
Moriya, K., Super-conformal surfaces associated with null complex holomorphic
  curves.
\newblock Bull. Lond. Math. Soc. \textbf{41}(2) (2009), 327--331.

\bibitem{Weierstrass1866}
Weierstass, K., Untersuchungen \"{u}ber die fl\"{a}chen, deren mittlere
  {K}r\"{u}mmung \"{u}berall gleich null ist.
\newblock Moatsber. Berliner Akad.  (1866) 612--625. 
\end{thebibliography}
\end{document}